\documentclass[journal]{IEEEtran}

\usepackage[utf8]{inputenc}
\usepackage[T1]{fontenc}
\usepackage{amsmath,amssymb,amsthm}
\usepackage{mathtools}
\usepackage{graphicx}
\usepackage{booktabs}
\usepackage{cite}
\usepackage{tikz}
\usetikzlibrary{arrows,positioning,calc}
\usepackage{comment}

\newtheorem{theorem}{Theorem}
\newtheorem{lemma}{Lemma}

\newtheorem{corollary}{Corollary}

\newtheorem{remark}{Remark}

\newcommand{\R}{\mathbb{R}}
\newcommand{\Sym}{\mathbb{S}}
\newcommand{\Hinf}{\mathcal{H}_\infty}
\newcommand{\norm}[1]{\left\|#1\right\|}

\newcommand{\blkdiag}{\mathrm{blkdiag}}

\newcommand{\T}{^\top}
\newcommand{\He}[1]{\mathrm{He}\left(#1\right)}
\newcommand{\Ell}{\mathcal{L}}

\newcommand{\flab}[1]{\label{fig:#1}}
\newcommand{\fig}[1]{Fig.~\ref{fig:#1}}

\newcommand{\elab}[1]{\label{eq:#1}}
\newcommand{\eqn}[1]{\eqref{eq:#1}}
\newcommand{\tlab}[1]{\label{tab:#1}}
\newcommand{\tab}[1]{Table~\ref{tab:#1}}

\newcommand{\bmtx}{\begin{bmatrix}}
\newcommand{\emtx}{\end{bmatrix}}

\newcommand{\vspacelem}{\vspace{0.2in}}

\title{Robust $\Hinf$ Observer Design via Finsler's Lemma and IQCs}

\author{Raktim~Bhattacharya and Felix~Biert\"umpfel
\thanks{R.~Bhattacharya is with the Department of Aerospace Engineering, Texas A\&M University, College Station, TX 77843, USA {\tt\small raktim@tamu.edu}.}
\thanks{F.~Biert\"umpfel is with the Chair of Flight Mechanics and Control at TU Dresden, 01307 Dresden, Germany, and the Department of Electrical Engineering and Computer Science, University of Michigan, Ann Arbor, MI 48109, USA. {\tt\small felix.biertuempfel@tu-dresden.de}, {\tt\small felixb@umich.edu}.}}

\begin{document}

\maketitle

\begin{abstract}
We formulate a robust $\Hinf$ observer design with integral quadratic constraints (IQCs) and block-structured uncertainty as a semidefinite program. The standard block-diagonal Lyapunov approach ($P = \blkdiag(P_1,P_2)$ with $Y = P_2 L$) couples the Lyapunov matrix and the IQC multiplier through a trade-off that causes infeasibility for wide uncertainty ranges. A structured Finsler slack separates the Lyapunov matrix from the observer gain with exact recovery, restoring feasibility with both scalar and block-diagonal multipliers. We demonstrate the effectiveness of the approach on the state estimation problem of a mass-spring-damper state estimation problem with $20$--$50$\% parametric uncertainty.
\end{abstract}


\section{Introduction}

Robust state estimation requires guaranteed performance under modeling error, parametric uncertainty, and sensor degradation. Various approaches based on Kalman filtering dominate practical applications and are effective when stochastic assumptions on the perturbations align with the operating regime. However, certifying worst-case estimation performance under deterministic uncertainty models requires a different framework.

This problem can be addressed by combining $H_\infty$ observer design and the integral quadratic constraint (IQC) framework. The former provides worst-case guarantees by bounding the induced $\mathcal{L}_2$ gain from exogenous inputs to the estimation error~\cite{nagpal1991filtering}. The IQCs encode structured uncertainty through multiplier-defined quadratic inequalities~\cite{megretski1997system}. Combining both frameworks yields a synthesis formulation in which the optimal observer gain and multipliers are calculated via semidefinite programming while exploiting block-diagonal uncertainty structures.


Building on Nagpal and Khargonekar~\cite{nagpal1991filtering}, $\Hinf$ filtering was extended to robust settings with norm-bounded uncertainty by Fu, de Souza, and Xie~\cite{fu1992hinf}, where the uncertainty is unstructured and the standard change-of-variables $W = PL$ linearizes the synthesis. When the uncertainty is block-structured, IQC multipliers can exploit the structure to reduce conservatism. IQC-based synthesis has been advanced by Scherer~\cite{scherer2001lpv}, and Veenman and Scherer~\cite{veenman2014iqc}, Wang et al.~\cite{Wang2016}, and Biert\"umpfel et al.~\cite{Biertuempfel2025}, primarily for controller design with dynamic multipliers. However, the application to observer design with static multipliers and block-structured uncertainty has received less attention.

The bilinear matrix inequality (BMI) occurring in robust observer synthesis is commonly resolved via a change-of-variables that imposes restrictive block-diagonal structures on the Lyapunov matrix~\cite{scherer1997multiobjective,duan2013lmis}. While polytopic and parameter-dependent Lyapunov functions~\cite{deoliveira1999extended} present alternatives that avoid the standard $\He{PA} \prec 0$ requirement, they typically scale exponentially with the number of uncertain parameters ($2^N$ vertices for $N$ parameters). Extended LMI characterizations using slack variables were introduced by de Oliveira et al.~\cite{deoliveira1999extended} and Pipeleers et al.~\cite{pipeleers2009extended} to decouple the Lyapunov and state matrices and lift these structural constraints.


The present paper addresses these limitations by applying Finsler's lemma~\cite{finsler1937,deoliveira2001finsler,skelton1998unified} to the IQC-augmented observer synthesis problem. By decoupling quadratic forms from linear constraints, we achieve polynomial scaling in $N$. 
We review the general $\Hinf$ and integral quadratic constraint framework in Section~\ref{sec:Pre}.
The key contributions of our work are presented in Section~\ref{sec:lmi_derivation} and are twofold:
\begin{enumerate}
    \item We prove (Lemma~\ref{lem:obstruction}) that for the block-diagonal Lyapunov structure $P = \blkdiag(P_{11},P_{22})$, the IQC multiplier $\Lambda$ tightens the Lyapunov condition $\He{P_{11}A} \prec 0$, making wide uncertainty ranges and large $\|C_q\|$ incompatible with feasibility (Remark~\ref{rem:lambda_tradeoff}).
    \item We derive a Finsler-based LMI (Theorem~\ref{thm:main_theorem}) with a structured slack that decouples the Lyapunov matrix from the observer gain, circumventing this trade-off with exact gain recovery.
\end{enumerate}
We apply the derived approach to a mass-spring-damper system state estimation under wide uncertainty in Section~\ref{sec:applications}.

\subsection*{Notation}
By $\R^{n\times m}$ real $n \times m$ matrices are denoted. Real symmetric matrices are denoted by $\mathbb{S}^{n}$. If such a matrix is positive definite we use the notation 
 $\mathbb{S}^n_{++}$. We use the short notation 
$\He{A}$ for the hermitian part of a square matrix $A$ defined by $A + A\T$. The notation
$X \preceq Y$ is used if $Y{-}X$ is positive semidefinite and $X \prec Y$ if $Y{-}X$ is positive definite. By $\blkdiag(A,\ldots,Z)$, a block-diagonal matrix composed of the matrices $A,\ldots,Z$ is denoted. The upper LFT, i.e., the feedback interconnection of a system $M$ and $\Delta$ is written as $\mathcal{F}_u(M,\Delta)$. Symmetric blocks in a matrix inequality are abbreviated with $\star$. The space of functions $p:[0,\infty)\rightarrow \R^n$ satisfying $\norm{p}_2 < \infty$ with $\norm{p}_2:=[\int_0^{\infty}p(t)^\top p(t)dt ]^{0.5}$ is denoted by $\Ell_2^n[0,\infty)$. Given $v:[0,\infty)\rightarrow \R^n$, $(v)_T$ is the truncated function: $(v)_T(t) = v(t)$ for $t\le T$ and $(v)_T(t) = 0$ otherwise. The extended space $\Ell_{2e}^n$ defines the set of functions $v$ such that $(v)_T \in \Ell_2\, \forall\,T\ge 0$.

\section{Preliminaries}\label{sec:Pre}
\subsection{Problem Formulation}\label{ss:prob}
\begin{figure}[h!]\label{fig:lft}
\centering
\begin{tikzpicture}[auto, thick, node distance=2cm, >=latex', scale=0.75, every node/.style={scale=0.75}]
    \node [draw, rectangle, minimum width=2.5cm, minimum height=2cm] (M) {$\mathcal{P}$};
    \node [draw, rectangle, minimum width=1.5cm, minimum height=1cm, above=0.8cm of M] (Delta) {$\Delta$};
    \draw [->] ([yshift=0.5cm]M.east) -- ++ (0.5,0) node[right]{$q$} |- (Delta.east);
    \draw [->] (Delta.west) -- ++ (-1,0) |- ([yshift=0.5cm]M.west) node[left, xshift=-0.5cm]{$p$};
    \draw [<-] (M.west) -- ++ (-1,0) node[left]{$w$};
    \draw [->] (M.east) -- ++ (1,0) node[right]{$z$};
    \draw [->] ([yshift=-0.5cm]M.east) -- ++ (1,0) node[right]{$y$};
\end{tikzpicture}
\caption{Upper LFT interconnection.}
\flab{lft}
\end{figure}
Consider the uncertain plant formulated by the feedback interconnection of a known linear time invariant system $\mathcal{P}$ and uncertainty $\Delta$. The short-hand notation of this interconnection is $\mathcal{F}_u(\mathcal{P},\Delta)$~\cite{duan2013lmis} which is also referred to as upper LFT. The known system $M$ has the state-space realization:
\begin{subequations}
\begin{align}
    \dot{x}(t) &= Ax(t) + B_p p(t) + B_w w(t), \elab{plant_state}\\
    q(t) &= C_q x(t) + D_{qp} p(t) + D_{qw} w(t), \elab{plant_q}\\
    z(t) &= C_z x(t), \elab{plant_z}\\
    y(t) &= C_y x(t) + D_{yp} p(t) + D_{yw} w(t),\elab{plant_y}
\end{align}
\elab{plant_dynamics}
\end{subequations}
\noindent where $x(t)\in \R^{n_x}$ is the state vector, $w(t)\in \R^{n_w}$ and $p(t)\in \R^{n_p}$ are the input vectors, and $q(t)\in \R^{n_q}$, $y(t)\in \R^{n_y}$, $z(t)\in \R^{n_z}$ are the output vectors at a time $t\in\R_{\ge0}$, respectively. The uncertainty feedback is $p = \Delta(q)$, where the uncertainty $\Delta:\Ell_{2e}^{n_q} \rightarrow \Ell_{2e}^{n_p}$ is a bounded and causal operator. The uncertainty is further confined to a set $\mathcal{D}$ of admissible uncertainties, i.e., $\Delta\in \mathcal{D}$.
%
\vspacelem
\begin{remark}
For systems with multiple uncertain parameters (e.g., $m$, $c$, $k$ in a mass-spring-damper), the uncertainty has block-diagonal structure:
\begin{equation}
    \Delta = \blkdiag(\delta_1 I_{n_1}, \delta_2 I_{n_2}, \ldots, \delta_N I_{n_N}), \quad |\delta_i| \leq 1
    \elab{delta_structure}
\end{equation}
where each $\delta_i$ is a real (repeated) normalized scalar uncertainty and $n_i$ is the number of times it appears in the realization.
Throughout, we denote the corresponding uncertainty set by
\begin{equation}
    \mathcal{D} := \big\{\blkdiag(\delta_1 I_{n_1},\ldots,\delta_N I_{n_N}) \;\big|\; |\delta_i|\le 1\big\}.
    \elab{delta_set}
\end{equation}
\end{remark}
\vspacelem
We make the following standing assumptions on the introduced systems:
\begin{enumerate}
    \item[(A1)] The nominal pair $(A,\,C_y)$ is detectable.
    \item[(A2)] The $\mathcal{F}_u(G,\Delta)$ is well-posed: $\det(I - D_{qp}\Delta) \neq 0$ for all $\Delta \in \mathcal{D}$.
    \item[(A3)] $D_{zp} = 0$ and $D_{zw} = 0$ (performance depends only on the state); $D_{yp}$ may be nonzero.
    \item[(A4)] The uncertain plant is robustly stable: $A + B_p \Delta (I - D_{qp}\Delta)^{-1} C_q$ is Hurwitz for all $\Delta \in \mathcal{D}$.
\end{enumerate}

For the known plant dynamics, a Luenberger observer can be defined:
\begin{align}
\dot{\hat{x}} &= A\hat{x} + L(y - C_y \hat{x}), \notag\\
 &= (A - LC_y)\hat{x} + L C_y x + L D_{yp} p + L D_{yw} w.
 \elab{observer_robust}
\end{align}
Defining the estimation error $e = x - \hat{x}$, the error dynamics are:
\begin{equation}
    \dot{e} = (A - LC_y)e + (B_p - L D_{yp}) p + (B_w - L D_{yw}) w.
    \elab{error_dynamics_robust}
\end{equation}
The performance error is $\tilde{z} = z - \hat{z} = C_z e$.
The uncertainty channel $q = C_q x + D_{qp}p + D_{qw}w$ depends on $x$, not $e$. Thus, the augmented state must contain both. Define $\xi := [x\T, e\T]\T$, $\eta := [p\T, w\T]\T$, and $\zeta:=[q\T, \tilde{z}\T]\T$. Using \eqn{plant_dynamics} and \eqn{error_dynamics_robust}, the augmented system is:
\begin{equation}
    \dot{\xi} = \underbrace{\begin{bmatrix} A & 0 \\ 0 & A - LC_y \end{bmatrix}}_{A_{\text{aug}}} \xi + \underbrace{\begin{bmatrix} B_p & B_w \\ B_p - LD_{yp} & B_w - LD_{yw} \end{bmatrix}}_{B_{\text{aug}}} \eta.
    \elab{augmented_dynamics}
\end{equation}

The outputs $\zeta = [q\T,\,\tilde{z}\T]\T$ are
\begin{equation}
    \zeta = \underbrace{\begin{bmatrix} C_q & 0 \\ 0 & C_z \end{bmatrix}}_{C_{\text{aug}}} \xi + \underbrace{\begin{bmatrix} D_{qp} & D_{qw} \\ 0 & 0 \end{bmatrix}}_{D_{\text{aug}}}\eta.
    \elab{augmented_outputs}
\end{equation}
The matrix $A_{\text{aug}}$ is block-diagonal. Thus, $x$ does not depend on $e$, and coupling enters only through $\eta$.

Our goal is to synthesize an observer gain $L$ and the smallest bound $\gamma>0$ such that the induced $\Ell_2$ gain from the exogenous input $w$ (disturbances and sensor noise) to the performance output $\tilde z = C_z e$ is bounded by $\gamma$ for every admissible uncertainty realization $\Delta$ satisfying \eqn{delta_structure}. Denoting by $T_{w\to \tilde z}(\Delta)$ the closed-loop transfer operator of the augmented interconnection (with $p=\Delta(q)$), the robust $\Hinf$ performance requirement is
\begin{equation}
    \sup_{\Delta\in\mathcal{D}}\;\norm{T_{w\to \tilde z}(\Delta)}_\infty < \gamma.
\end{equation}
The following subsections review key ingredients for our derivations of an efficient and less conservative robust observer design using integral quadratic constraints.

\subsection{Classical $\Hinf$ Observer Design}\label{sec:nominal_observer}

Consider the nominal linear time-invariant plant $\mathcal{P}$ with $\Delta = 0$ and thus $p=0$ defined by
\begin{subequations}
\begin{align}
 \dot{x} &= Ax + B_w w \\
 z &= C_z x\\
 y &= C_y x + D_{yw} w
\end{align}    
\end{subequations}
The Luenberger observer
\begin{equation}
    \dot{\hat{x}} = A\hat{x} + L(y - C_y\hat{x})
\end{equation}
produces the error dynamics
\begin{subequations}
\begin{align}
    \dot{e} &= (A - LC_y)e + (B_w - LD_{yw})w\\ \tilde{z} &= C_z e,
\end{align}
    \elab{error_dynamics_nom}
\end{subequations}
and the corresponding $\Hinf$ synthesis problem results in the following optimization:
\begin{equation}\label{eq:OptiNom}
\begin{split}
\min_{L} \quad & \gamma \\
\mathrm{s.t.} \quad & \|C_z(sI - A + LC_y)^{-1}(B_w - LD_{yw})\|_\infty < \gamma.
\end{split}
\end{equation}
We invoke the well-known bounded real lemma (BRL) to convert~\eqref{eq:OptiNom} to a tractable linear matrix inequality (LMI). This is formalized in Lemma~\ref{lem:nomBRL} \cite{nagpal1991filtering,fu1992hinf}.
\begin{lemma}[$\Hinf$ Observer Design]\label{lem:nomBRL}
    There exists an observer gain $L$ such that the error dynamics~\eqref{eq:error_dynamics_nom} are asymptotically stable and satisfy the performance bound in~\eqref{eq:OptiNom} if and only if there exists $P \in \Sym_{++}^{n_x}$ and a matrix $W\in \R^{n_x \times n_y}$ such that
\begin{equation}
    \begin{bmatrix}
        A\T P + PA - C_y\T W\T - WC_y \!\!&\!\! PB_w - WD_{yw} \!\!&\!\! C_z\T \\
        \star &\!\! -\gamma I & 0 \\
        \star&\!\! \star \!\!&\!\! -\gamma I
    \end{bmatrix} \!\preceq\! 0
    \elab{observer_lmi}
\end{equation}
The observer gain $L$ is recovered from feasible values of $P$ and $W$ as $L = P^{-1}W$.
\end{lemma}
\begin{proof}
    The proof follows standard dissipativity arguments applied to the (scaled) Lyapunov function $V(x) = x\T P x$. It is omitted for brevity and can be found in, e.g., \cite{fu1992hinf}.
\end{proof}

\begin{remark}
In nominal $\Hinf$ observer design (Lemma~\ref{lem:nomBRL}), the Bounded Real Lemma (BRL) involves only the error dynamics $(A - LC_y, B_w - LD_{yw}, C_z)$. The robust design requires the augmented state $\xi = [x;\,e]$ because the uncertainty channel, $q = C_q x$ depends on the plant state $x$, not the error $e$. Carrying $x$ in $\xi$ introduces $\He{P_{11}A}$ into the Lyapunov condition --- a term the observer gain $L$ cannot influence, since $L$ enters only the $e$-block (Lemma~\ref{lem:obstruction}).
\end{remark}
\subsection{Integral Quadratic Constraints (IQCs)}\label{sec:iqc}
The IQC framework encodes structured uncertainty through quadratic inequalities on interconnection signals~\cite{megretski1997system,pfifer2015iqc}. The IQC bounds the input-output behavior of the uncertainty and effectively removes the explicit uncertainty model from $\mathcal{F}_u(M,\Delta)$. We use static \emph{hard} IQCs to upper bound the input-output behavior of an uncertainty $\Delta\in\mathcal{D}$. The uncertainty $\Delta$ satisfies the hard IQC defined by a matrix $M\in \Sym^{n_q+n_p}$ if, for every $T > 0$,
\begin{equation}
    \int_0^T \begin{bmatrix} q(t) \\ p(t) \end{bmatrix}\T M \begin{bmatrix} q(t) \\ p(t) \end{bmatrix} \, dt \ge 0.
    \elab{iqc_time}
\end{equation}
\vspace{5pt}

A suitable diagonal multiplier $M$ for block-diagonal uncertainty $\Delta = \blkdiag(\Delta_1,\ldots,\Delta_N)$ with $\|\Delta_i\| \le 1$ is defined by
\begin{equation}
    M(\Lambda) = \begin{bmatrix} \Lambda & 0 \\ 0 & -\Lambda \end{bmatrix}, \quad \Lambda = \blkdiag(\Lambda_1,\ldots,\Lambda_N),\ \Lambda_i \succ 0.
    \elab{pi_lambda}
\end{equation}
This parameterization assigns independent scalings to each uncertainty block, reducing conservatism compared to a single $\Lambda = \lambda I$. When the uncertainty blocks are non-repeated ($n_i = 1$), $\Lambda_i$ is a scalar. When $n_i > 1$ (repeated scalar uncertainty $\Delta_i = \delta_i I_{n_i}$), $\Lambda_i \in \mathbb{S}^{n_i}_{++}$ provides additional degrees of freedom.

Note that the hard IQC~\eqn{iqc_time} requires $\Lambda \succeq \Delta\T \Lambda \Delta$ for all $\Delta \in \mathcal{D}$. The block-diagonal structure $\Lambda = \blkdiag(\Lambda_1,\ldots,\Lambda_N)$ is the richest multiplier satisfying this condition for general norm-bounded uncertainties, since $\Delta_i\T \Lambda_i \Delta_i = \delta_i^2 \Lambda_i \preceq \Lambda_i$ for $|\delta_i| \leq 1$. Richer multiplier classes (e.g., $DG$-like multipliers for strictly real uncertainties \cite{scherer2001lpv}) use a different multiplier structure and can reduce conservatism further. Moreover, the use of block-diagonal multipliers provides structural advantages for the derivation of the robust $\Hinf$ observer in Section~\ref{sec:lmi_derivation}.

\section{LMI Derivation of Robust $\Hinf$ Observer}\label{sec:lmi_derivation}
The $\Hinf$-observer problem can be posed in the form of a dissipation inequality (see Section~\ref{sec:nominal_observer}). This fact allows for an extension of the BRL using hard IQCs (see, e.g.,~\cite{Seiler2015}), which we sketch below. Introducing a storage function $V(\xi)=\xi\T P\xi$ with $P \succ 0$, the pointwise dissipation inequality
\begin{equation}
    \dot{V}(\xi) + \tilde{z}\T \tilde{z} - \gamma^2 w\T w + \begin{bmatrix} q \\ p \end{bmatrix}\T M \begin{bmatrix} q \\ p \end{bmatrix} \le 0
    \elab{iqc_dissipation}
\end{equation}
integrates over $[0,T]$ to yield the $\Hinf$ certificate $\int_0^T \|\tilde{z}\|^2\,dt \le \gamma^2 \int_0^T \|w\|^2\,dt$ for zero initial conditions, using $V(\xi(T)) \ge 0$ and the hard IQC~\eqn{iqc_time}.
To evaluate this inequality in state-space, we use the augmented state $\xi = [x\T,\,e\T]\T$ with $e = x - \hat{x}$, giving $A_{\text{aug}} = \blkdiag(A,\,A-LC_y)$ throughout the section.

\subsection{Structural Limitations of the Classical Formulation}
Substituting $\dot{V}(\xi) = \xi\T(\He{PA_{\text{aug}}})\xi + 2\xi\T P B_{\text{aug}} \eta$ into \eqn{iqc_dissipation} and collecting all terms as a quadratic form in $[\xi\T,\,\eta\T]\T$ yields a matrix inequality that is \emph{bilinear} in $(P,\,L)$. In other words, the product $PA_{\text{aug}}(L)$ couples the Lyapunov matrix with the observer gain, since $L$ enters $A_{\text{aug}}$ and $B_{\text{aug}}$ through the error dynamics. With $P = \blkdiag(P_{11},P_{22})$, the substitution $Y = P_{22}L$ linearizes the entire inequality, but imposes the structural requirement $\He{P_{11}A} \prec 0$ (Section~\ref{sec:blkdiag}).

The requirement $\He{P_{11}A} \prec 0$ is well known for block-diagonal Lyapunov approaches in robust synthesis~\cite{scherer1997multiobjective,deoliveira1999extended}. The following lemma characterizes the interaction between $\Lambda$ and this condition.

\begin{lemma}[Multiplier--Lyapunov constraint for blkdiag $P$]\label{lem:obstruction}
Let $P = \blkdiag(P_{11},P_{22})$ with $P_{11},P_{22} \succ 0$, and let $\Lambda \succeq 0$. If the IQC-extended BRL condition
\begin{equation}
    \begin{bmatrix}
        \He{P A_{\text{aug}}} + Q_{22} & P B_{\text{aug}} + Q_{23} \\
        \star & Q_{33}
    \end{bmatrix} \prec 0
    \elab{brl_iqc}
\end{equation}
holds, then $\He{P_{11}A} + C_q\T\Lambda C_q \prec 0$, and consequently $\He{P_{11}A} \prec 0$.
\end{lemma}
\begin{proof}
With $P = \blkdiag(P_{11},P_{22})$ and $A_{\text{aug}} = \blkdiag(A,\,A-LC_y)$ we can write
\begin{equation*}
    \He{P A_{\text{aug}}} = \blkdiag\!\big(\He{P_{11}A},\;\He{P_{22}(A-LC_y)}\big).
\end{equation*}
The $(1,1)$ block of $Q_{22} = C_{\text{aug}}\T \Pi_\Lambda C_{\text{aug}}$ is $C_q\T\Lambda C_q \succeq 0$ (since $\Lambda \succeq 0$ and $C_{\text{aug}} = [C_q,\,0;\,0,\,C_z]$). The leading $n\times n$ block of \eqn{brl_iqc} satisfies
\begin{equation*}
    \He{P_{11}A} + C_q\T \Lambda C_q + S \prec 0,
\end{equation*}
where $S \succeq 0$ collects the Schur complement contribution from the remaining blocks. Since, $S \succeq 0$ and $C_q\T\Lambda C_q \succeq 0$ it strictly follows that $\He{P_{11}A}\prec 0$.
\end{proof}

The condition $\He{P_{11}A} \prec 0$ is classical~\cite{scherer1997multiobjective,deoliveira1999extended}. Lemma~\ref{lem:obstruction} adds that the IQC multiplier acts against feasibility in the plant-state block. In other words, increasing $\Lambda$ to tighten the uncertainty characterization simultaneously tightens the Lyapunov constraint on $A$. 
This artificially restricts how much uncertainty the observer can handle by introducing significant conservatism. We will quantify this trade-off in Section~\ref{sec:applications} (Remark~\ref{rem:lambda_tradeoff}).

\subsection{Resolution via Finsler's Lemma with Structured Slack}
To overcome the structural obstruction in the LMI~\eqref{eq:brl_iqc} identified by Lemma~\ref{lem:obstruction}, the Lyapunov matrix $P$ must be decoupled from the system dynamics.
We resolve the BMI by applying Finsler's lemma with a \emph{structured} slack matrix that provides a single substitution variable, yielding exact gain recovery without a posteriori verification.
\begin{lemma}[Finsler~\cite{finsler1937}]
\label{lem:finsler}
Let $Q = Q\T \in \R^{m \times m}$ and $\mathcal{B} \in \R^{k \times m}$. The following statements are equivalent:
\begin{enumerate}
    \item $v\T Q v < 0$ for all $v \neq 0$ satisfying $\mathcal{B} v = 0$.
    \item There exists $G \in \R^{m \times k}$ such that $Q + G\mathcal{B} + \mathcal{B}\T G\T \prec 0$.
\end{enumerate}
\end{lemma}
\begin{proof}
    The proof is provided in~\cite{finsler1937}.
\end{proof}
\noindent To leverage Finsler's lemma, we first introduce selection matrices for $q$, $\tilde{z}$, $p$, $w$:
\begin{equation}
    M_q := \begin{bmatrix} I_{n_q} & 0 \end{bmatrix},\quad
    M_{\tilde{z}} := \begin{bmatrix} 0 & I_{n_z} \end{bmatrix} \in \R^{\cdot\times(n_q+n_z)},
    \elab{selection_matrices}
\end{equation}
\begin{equation}
    E_p := \begin{bmatrix} I_{n_p} & 0 \end{bmatrix},\quad
    E_w := \begin{bmatrix} 0 & I_{n_w} \end{bmatrix} \in \R^{\cdot\times n_\eta}
    \elab{eta_selection}
\end{equation}
We then set $\Pi_\Lambda := M_{\tilde{z}}\T M_{\tilde{z}} + M_q\T \Lambda M_q$ and define
\begin{align}
    Q_{22} &:= C_{\text{aug}}\T \Pi_\Lambda C_{\text{aug}},\quad
    Q_{23} := C_{\text{aug}}\T \Pi_\Lambda D_{\text{aug}}, \notag\\
    Q_{33} &:= D_{\text{aug}}\T \Pi_\Lambda D_{\text{aug}} - \gamma^2 E_w\T E_w - E_p\T \Lambda E_p. \elab{Q33}
\end{align}

Next, we introduce a structured slack variable to facilitate a single substitution in the robust observer synthesis LMI, which avoids the typical relaxation gap. By setting $\nu := [\dot{\xi}\T,\,\xi\T,\,\eta\T]\T$ and $\mathcal{B} := [I,\,-A_{\text{aug}},\,-B_{\text{aug}}]$, the dissipation inequality~\eqn{iqc_dissipation} becomes $\nu\T\mathcal{Q}\,\nu \le 0$ on $\ker\mathcal{B}$. By Lemma~\ref{lem:finsler}, this holds if and only if $\mathcal{Q} + G\mathcal{B} + \mathcal{B}\T G\T \prec 0$ for some $G \in \R^{(2n_\xi+n_\eta)\times n_\xi}$.

A general (unstructured) $G = [G_1;\,G_2;\,G_3]$ requires three independent substitutions $\mathcal{Y}_i = G_{i2}L$. Since the three substitutions need not imply the same $L$, this introduces a relaxation gap. Instead, we impose the structure
\begin{equation}
    G = \begin{bmatrix} Z \\ aZ \\ Z_3 \end{bmatrix},
    \elab{structured_slack}
\end{equation}
where $Z \in \R^{n_\xi \times n_\xi}$ is a free (not necessarily symmetric) matrix, $a > 0$ is a fixed scalar parameter, and $Z_3 \in \R^{n_\eta \times n_\xi}$ has columns $n{+}1{:}2n$ constrained as $Z_{3e} = \beta Z_{e,\text{bot}}$ for a fixed $\beta > 0$. Here $Z_{e,\text{bot}} \in \R^{n\times n}$ denotes the submatrix of $Z$ at rows $n{+}1{:}2n$, columns $n{+}1{:}2n$.

This structure yields a \emph{single} substitution
\begin{equation}
    W := Z_{e,\text{bot}}\, L \in \R^{n\times n_y},
    \elab{single_substitution}
\end{equation}
which linearizes all appearances of $L$ in $A_{\text{aug}}$ and $B_{\text{aug}}$ simultaneously. Denoting by $Z_e$ columns $n{+}1{:}2n$ of $Z$, the terms in $A_{\text{aug}}$ become $Z_e(A - LC_y) = Z_eA - WC_y$, while in $B_{\text{aug}}$ they become $Z_e(-LD_{yp},\,-LD_{yw}) = -W(D_{yp},\,D_{yw})$. The constrained $Z_3$ block uses $Z_{3e}L = \beta Z_{e,\text{bot}}L = \beta W$, preserving the single-substitution property.

\begin{theorem}[IQC-based $\Hinf$ Observer via Finsler's Lemma]
\label{thm:main_theorem}
Consider the uncertain plant \eqn{plant_dynamics} under assumptions (A1)--(A4), with augmented dynamics \eqn{augmented_dynamics}--\eqn{augmented_outputs}, block-diagonal uncertainty $\Delta\in\mathcal{D}$, and hard IQC multiplier $\Pi(\Lambda)$, $\Lambda_i \succ 0$. Suppose there exist $P = P\T \succ 0 \in \R^{n_\xi\times n_\xi}$, a structured slack $G$ as in \eqn{structured_slack} with $Z \in \R^{n_\xi\times n_\xi}$, $Z_3 \in \R^{n_\eta \times n_\xi}$, substitution $W \in \R^{n\times n_y}$ as in \eqn{single_substitution}, and $\gamma > 0$, such that
\begin{equation}
    \He{Z_{e,\text{bot}}} \prec -\varepsilon_Z I,\quad \varepsilon_Z > 0,
    \elab{Z_constraint}
\end{equation}
and the $(2n_\xi+n_\eta)\times(2n_\xi+n_\eta)$ LMI
\begin{equation}
\elab{iqc_observer_lmi}
    \footnotesize
    \begin{bmatrix}
        \He{Z} & P - Z A_{\text{aug}}^W + aZ\T & -Z B_{\text{aug}}^W + Z_3\T \\
        \star & Q_{22} - a\He{Z A_{\text{aug}}^W} & Q_{23} - aZ B_{\text{aug}}^W - (Z_3 A_{\text{aug}}^W)\T \\
        \star & \star & Q_{33} - \He{Z_3 B_{\text{aug}}^W}
    \end{bmatrix} \prec 0
\end{equation}
holds, where $A_{\text{aug}}^W$ and $B_{\text{aug}}^W$ denote $A_{\text{aug}}$ and $B_{\text{aug}}$ with $Z_eL$ replaced by $W$ (and $Z_{3e}L$ by $\beta W$) as described above. Then:
\begin{enumerate}
    \item[\emph{(a)}] Condition~\eqn{Z_constraint} implies $Z_{e,\text{bot}}$ is nonsingular, so the observer gain $L = Z_{e,\text{bot}}^{-1}W$ is well-defined.
    \item[\emph{(b)}] For any $\Delta$ satisfying the hard IQC~\eqn{iqc_time} with multiplier $\Pi(\Lambda)$, the Luenberger observer $\dot{\hat{x}} = A\hat{x} + L(y - C_y\hat{x})$ achieves $\norm{T_{w\to\tilde{z}}(\Delta)}_\infty < \gamma$.
\end{enumerate}
\end{theorem}

\begin{proof}
\textit{Part~(a).} $\He{Z_{e,\text{bot}}} \prec 0$ implies $Z_{e,\text{bot}}v \neq 0$ for all $v \neq 0$, so $Z_{e,\text{bot}}$ is nonsingular.

\textit{Part~(b).} The dissipation~\eqn{iqc_dissipation} is $\nu\T\mathcal{Q}\,\nu \le 0$ on $\ker\mathcal{B}$, where $\nu = [\dot{\xi}\T,\,\xi\T,\,\eta\T]\T$, $\mathcal{B} = [I,\,-A_{\text{aug}},\,-B_{\text{aug}}]$, and
\begin{equation*}
    \mathcal{Q} = \begin{bmatrix} 0 & P & 0 \\ P & Q_{22} & Q_{23} \\ 0 & Q_{23}\T & Q_{33} \end{bmatrix}.
\end{equation*}
By Lemma~\ref{lem:finsler}, this holds if and only if $\mathcal{Q} + G\mathcal{B} + \mathcal{B}\T G\T \prec 0$ for some $G$. Substituting $G = [Z;\,aZ;\,Z_3]$ and expanding gives \eqn{iqc_observer_lmi}. The observer gain $L$ enters $A_{\text{aug}}$ and $B_{\text{aug}}$ only through columns $n{+}1{:}2n$ (the error block). With $W = Z_{e,\text{bot}}L$ as in \eqn{single_substitution}, all products involving $L$ become linear. At the recovered $L = Z_{e,\text{bot}}^{-1}W$, these substitutions hold exactly, so \eqn{iqc_observer_lmi} is satisfied at the true $(A_{\text{aug}},B_{\text{aug}})$. The $\Hinf$ bound follows from integrating the dissipation with $\xi(0) = 0$, $V(\xi(T)) \ge 0$, and the hard IQC~\eqn{iqc_time}.
\end{proof}

Since $L = Z_{e,\text{bot}}^{-1}W$, rather than the traditional $P^{-1}W$, the Lyapunov matrix $P$ is decoupled from the gain recovery. This property relaxes the $\Lambda$--$P$ coupling that causes infeasibility in Lemma~\ref{lem:obstruction}. 

Robust stability follows from (A4) and the strict LMI. Integrating \eqn{iqc_dissipation} with $w=0$ guarantees the exponential decay of $V(\xi)$.

The parameters $a > 0$ and $\beta > 0$ control the slack structure; the SDP is solved for each $(a,\beta)$ on a coarse grid and the pair minimizing $\gamma$ is selected. In Section~\ref{sec:applications}, $a = 1$ and $\beta = 0.01$ are used.

\begin{remark}[Connection to the IQC theorem]
The Megretski--Rantzer IQC theorem~\cite{megretski1997system} requires the nominal interconnection ($\Delta = 0$) to be stable. This implies that $A_{\text{aug}}|_{\Delta=0} = \blkdiag(A,\,A-LC_y)$ must be Hurwitz, which follows directly from (A4) and gain recovery. Since we use hard IQCs with static multipliers and a quadratic storage function, the homotopy condition is automatically satisfied~\cite{pfifer2015iqc}.
\end{remark}

\begin{remark}[IQC certificate validity]\label{rem:iqc_gap}
Theorem~\ref{thm:main_theorem} certifies $\gamma$ for all $\Delta$ satisfying the hard IQC~\eqn{iqc_time}. The certificate is valid over $\mathcal{D}$ only if $\Lambda \succeq \Delta\T\Lambda\Delta$ holds for all $\Delta \in \mathcal{D}$, which requires $\Lambda = \blkdiag(\Lambda_1,\ldots,\Lambda_N)$ matching the block structure of $\Delta$. In general, a full-block multiplier $\Lambda \in \mathbb{S}^{n_q}_{++}$ violates this condition and can produce an artificially low $\gamma$ smaller than the actual worst-case.
\end{remark}

\subsection{Block-Diagonal Specialization}\label{sec:blkdiag}

With $P = \blkdiag(P_{11},P_{22})$ and $Y := P_{22}L$, the null-space form of Lemma~\ref{lem:finsler} (condition~1) eliminates the slack $G$ entirely. This yields a simpler SDP but requires $\He{P_{11}A} \prec 0$.
\begin{corollary}[Block-diagonal specialization]\label{cor:blkdiag}
Let $P = \blkdiag(P_{11},P_{22})$ and $Y := P_{22}L$. If $P$, $Y$, $\Lambda_i \succ 0$, and $\gamma > 0$ satisfy the LMI
\begin{equation}
    \begin{bmatrix}
        \He{P A_{\text{aug}}} + Q_{22} & P B_{\text{aug}} + Q_{23} \\
        \star & Q_{33}
    \end{bmatrix} \prec 0,
    \elab{blkdiag_lmi}
\end{equation}
where the substitution $Y = P_{22}L$ linearizes the terms $P A_{\text{aug}} = \blkdiag(P_{11}A,\, P_{22}A - YC_y)$ and $PB_\text{aug}$, then $L = P_{22}^{-1}Y$ achieves $\norm{T_{w\to\tilde{z}}(\Delta)}_\infty < \gamma$ for all $\Delta$ satisfying the hard IQC~\eqn{iqc_time}. Since $Q_{22} \succeq 0$, feasibility requires $\He{P_{11}A} \prec 0$.
\end{corollary}

This specialization can become infeasible for wide uncertainty ranges due to the $\Lambda$--$P$ conservatism trade-off. Remark~\ref{rem:lambda_tradeoff} in Section~\ref{sec:applications} quantifies this trade-off for the considered numerical example.

\section{Numerical Example}\label{sec:applications}
The numerical example considers a mass-spring-damper (MCK) system. The system is defined by the state space representation
\begin{align}\label{eq:msd}
    \dot{x} &= \begin{bmatrix} 0 & 1 \\ -\frac{k}{m} & -\frac{c}{m} \end{bmatrix}\bmtx x_1 \\x_2 \emtx  + \begin{bmatrix} 0 \\ \frac{1}{m} \end{bmatrix} w\\
    y &= \begin{bmatrix} -\frac{k}{m}  & -\frac{c}{m}\end{bmatrix}\bmtx x_1 \\ x_2\emtx + \frac{1}{m} w,
\end{align}
where $x_1$ is the displacement, $x_2$ is the velocity and $y$ the measured acceleration.
In~\eqref{eq:msd}, $m$ denotes the mass with nominal value $m_0 = 1$. The damping $c$ has a nominal value $c_0=0.5$. The spring constant $k$ has a nominal value $k_0$ of $2$. All system parameters are assumed uncertain with uncertainty ranges $m \in [0.8,\,1.2]$, $c \in [0.3,\,0.8]$, $k \in [1.5,\,2.6]$.
The system is Hurwitz at nominal parameters ($\zeta = 0.18$). The LFT yields three scalar uncertainty blocks ($\delta_m$, $\delta_c$, $\delta_k$) with $n_p = n_q = 5$ uncertainty channels and $\|C_q\| = 0.86$.

\tab{mck_results} compares formulation--multiplier combinations. Corollary~\ref{cor:blkdiag} is infeasible with both scalar and block-diagonal $\Lambda$ due to the $\Lambda$--$P$ trade-off (Remark~\ref{rem:lambda_tradeoff}). The structured slack of Theorem~\ref{thm:main_theorem} restores feasibility and achieves $\gamma = 0.838$ with scalar $\Lambda$ and $\gamma = 0.704$ with block-diagonal $\Lambda = \blkdiag(\Lambda_1,\ldots,\Lambda_N)$, $\Lambda_i \in \mathbb{S}^{n_i}_{++}$. The block-diagonal structure is the richest multiplier satisfying $\Lambda \succeq \Delta\T\Lambda\Delta$ over $\mathcal{D}$ (Remark~\ref{rem:iqc_gap}). 

\begin{table}[htbp]
\centering
\caption{MCK observer synthesis ($n_p = n_q = 5$)}
\tlab{mck_results}
\begin{tabular}{lcl}
\toprule
\textbf{Formulation} & $\gamma$ & \textbf{Multiplier} \\
\midrule
Nominal (no unc.) & 0.500 & --- \\
Corollary~\ref{cor:blkdiag} (scalar $\Lambda$) & infeasible & $\Lambda = \blkdiag(\lambda_i I)$ \\
Corollary~\ref{cor:blkdiag} (blkdiag $\Lambda$) & infeasible & $\Lambda = \blkdiag(\Lambda_i)$ \\
Theorem~\ref{thm:main_theorem} (scalar $\Lambda$) & 0.838 & $\Lambda = \blkdiag(\lambda_i I)$ \\
Theorem~\ref{thm:main_theorem} (blkdiag $\Lambda$) & 0.704 & $\Lambda = \blkdiag(\Lambda_i)$ \\
\bottomrule
\end{tabular}
\end{table}

\fig{mck_error} shows Monte Carlo error norm trajectories over 50 random parameter realizations. \fig{mck_certificate} validates the certificates by computing $\|T_{w \to \tilde{z}}\|_\infty$ over a dense parameter grid ($27{,}000$ samples).

\begin{figure}[htbp]
    \centering
    \includegraphics[width=0.95\columnwidth]{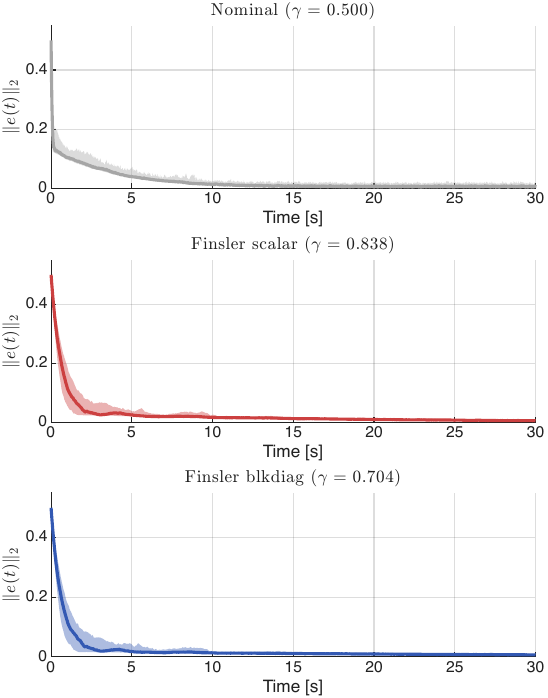}
    \caption{MCK: $\|e(t)\|_2$ over 50 random parameter realizations (shaded: 5--95\%, solid: median). Each subplot shows one formulation with its certified $\gamma$.}
    \flab{mck_error}
\end{figure}

\begin{figure}[htbp]
    \centering
    \includegraphics[width=0.95\columnwidth]{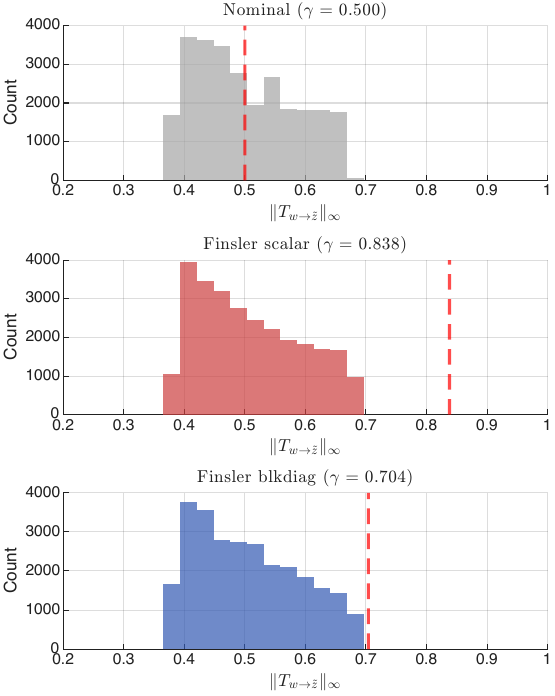}
    \caption{MCK certificate validation: $\Hinf$ norm distribution over a dense parameter grid ($27{,}000$ samples). Dashed red lines show certified $\gamma$. Both Finsler certificates (scalar and block-diagonal $\Lambda$) are valid over $\mathcal{D}$.}
    \flab{mck_certificate}
\end{figure}

\subsubsection*{Discussion}
The nominal design's worst-case $\Hinf$ norm ($0.67$, \fig{mck_certificate}) exceeds its certified $\gamma = 0.50$. Both Finsler certificates are valid over $\mathcal{D}$: $\gamma = 0.838$ with scalar $\Lambda$ and $\gamma = 0.704$ with block-diagonal $\Lambda$. The block-diagonal multiplier provides a tighter bound by cross-weighting the repeated uncertainty channels ($n_3 = 3$ for the $k$-block). The block diagonal Lyapunov specialization remains infeasible for both multiplier structures.

\begin{remark}[Multiplier--Lyapunov trade-off]\label{rem:lambda_tradeoff}
The block-diagonal LMI requires $\Lambda$ large enough for $Q_{33} \prec 0$ yet small enough that $C_q\T \Lambda C_q$ does not overwhelm $\He{P_{11}A}$. Large uncertainty magnitudes produce large $\|C_q\|$ rendering these requirements incompatible. In the considered MCK example ($\|C_q\| = 0.86$), the block-diagonal specialization with scalar $\Lambda$ is infeasible despite a nominal stability margin of $\min_i|{\rm Re}(\lambda_i)| = 0.25$. The structured Finsler slack relaxes this coupling and yields significantly less conservative results.
\end{remark}
\section{Conclusions}

The block-diagonal Lyapunov structure for IQC-based robust observer synthesis requires $\He{P_{11}A} \prec 0$, a condition tightened by the multiplier~$\Lambda$ (Lemma~\ref{lem:obstruction}). A structured Finsler slack (Theorem~\ref{thm:main_theorem}) avoids this coupling. On a mass-spring-damper with 20--50\% parametric uncertainty, the Finsler formulation is feasible ($\gamma = 0.704$) where the block-diagonal approach is not. The LMI has size $2n_\xi{+}n_\eta$ and is solved in under one second. Future work will consider the extension to general dynamic IQC multipliers. Here, standard congruences for filter synthesis introduce a bilinear coupling between the Lyapunov auxiliary and the IQC multiplier that the observer's substitution $Y = P_{22}L$ does not encounter.


\section{Acknowledgments}
The first author acknowledges the support by the AFOSR grant FA9550-22-1-0539 with Dr. Erik Blasch as the program director. The second author acknowledges funding by the European Union under Grant No. 101153910. Views and opinions expressed are those of the authors and do not necessarily reflect those of AFOSR or the European Union.

\bibliographystyle{IEEEtran}
\bibliography{references}

\end{document}